\documentclass{article}

\usepackage{amsmath, amsthm, amsfonts}
\usepackage[utf8]{inputenc}
\usepackage[T1]{fontenc}
\usepackage{amssymb}
\usepackage{eurosym}

\newtheorem{thm}{Theorem}[section]

\theoremstyle{definition}
\newtheorem{defn}[thm]{Definition}
\theoremstyle{remark}

\theoremstyle{proof}

\def\JJ{\mathbb{J}}
\def\SS{\mathbb{S}}

\title{\bf{Hyperbolic Jacobsthal Spinor Sequences and Their Mathematical Properties}}
\author{Selime Beyza \"{O}Z\c{C}EVİK and Abdullah DERTLİ\\
  \small  Department of Mathematics\\
  \small Ondokuz May\i s University\\
  \small Turkey\\
  \small ozcevikbeyza8@gmail.com, abdullah.dertli@gmail.com
}

\author{Selime Beyza \"{O}Z\c{C}EVİK and Abdullah DERTLİ\\
\small  Department of Mathematics\\
  \small Ondokuz May\i s University\\
  \small Turkey\\
  \small ozcevikbeyza8@gmail.com, abdullah.dertli@gmail.com
}

\begin{document}
\maketitle

\abstract{In this study, novel Hyperbolic spinor sequences of  Jacobsthal, Jacobsthal-Lucas and Jacobsthal polynomial, which have not been studied before, are defined by investigating the relationship between spinors, which are important mathematical objects used in physics and mathematics, and split Jacobsthal and split Jacobsthal-Lucas quaternions, which are extensions of the known Jacobsthal and Jacobsthal-Lucas numbers to quaternion algebra. The recurrence relations of sequences whose members are Hyperbolic Jacobsthal, Jacobsthal-Lucas and Jacobsthal polynomial spinors are described. Additionally, certain properties of these spinors, such as the generator function and Binet formula, are presented and some identities resulting from these spinors are obtained.\\
{\bf{MSC:}} 11R52, 11B37. \\
{\bf{Keywords:}} Split quaternion, Hyperbolic spinor, Jacobsthal numbers.}

\section{Introduction}
The concept of spinors can be approached from two different angles: a physical perspective and a mathematical perspective. From a mathematical perspective, spinors were first discussed by Cartan in his 1913 paper when he studied linear representations of simple groups \cite{cartan1913groupes}. In 1927, Pauli introduced his famous spin matrices and made the first application of spinors in the field of physics \cite{pauli1927quantenmechanik}. In 1928, Dirac discovered the relativistic theory of electron spins and explained the connection between spinors and Lorentz groups \cite{dirac1928quantum}, establishing the first connection between physics and mathematics for spinors. Between 1929 and 1932, Van der Verden's studies analyzed this connection mathematically \cite{waerden1929,van1932gruppentheoretische}. In 1935, Cartan's spinor theory was extended to n dimensions \cite{brauer1935spinors}. The most important work in the literature on the applications of spinors in physics is undoubtedly the work of Fierz, in which he studied the relativistic theory of force-independent particles with arbitrary spin \cite{fierz1939relativistische}. Spinor theory can be considered as a part of the theory of group representations. There are various sources in the literature on this subject.

The connection between the spinor wave equations and the Lorentz groups was first investigated by Bargmann and Wigner in their 1948 work \cite{bargmann1948group}. This perspective was further developed by Joos in 1962 \cite{joos1962darstellungstheorie}, but it was Weinberg who strengthened this perspective with his work in 1964 and 1965 \cite{weinberg1964feynman,weinberg1965comments}. Since then, spinors have been involved in important applications in the quantum theory of physics, atomic physics, general relativity, particle theory, and quantum mechanics. An extensive bibliography on the applications of spinors in physics can be found in the work of Parke and Jehle \cite{parkejehle1965boulder}. Additionally, Cartan investigated spinor theory by giving a purely geometric description of spinors in mathematical terms \cite{cartan2012theory}. Thanks to these geometrical foundations, spinors were used by physicists in quantum mechanics.

The relationship between spinors and quaternions was established by Vivarelli \cite{vivarelli1984development}. After the introduction of quaternions by Hamilton in 1843, James Cockle defined split quaternions from a different perspective \cite{cockle1849lii}. Split quaternions are elements of a 4-dimensional associative algebra. They have zero divisors, nilpotent elements, and nontrivial idempotent elements. In this study, we defined Hyperbolic Jacobsthal, Jacobsthal-Lucas and Jacobsthal polynomial spinors, building on the above studies and studied the features characterizing these spinors.

\section{Preliminaries}
In this section, we will give basic definitions and theorems about split quaternions and hyperbolic spinors. Let us consider the vector $\left(\alpha_1, \alpha_2, \alpha_3\right)$ with $x_1^2+x_2^2-x_3^2=0$ in the three dimensional Minkowski space $\mathbb{R}_1^3$. These vectors form a two-dimensional surface in the twodimensional subspace of  $\mathbb{R}_1^3$. If the parameters of this two-dimensional surface are taken as $\varphi_1$ and $\varphi_2$, the following equations can be written
$$
\begin{aligned}
& \alpha_1=\varphi_1^2-\varphi_2^2, \\
& \alpha_2=j\left(\varphi_1^2+\varphi_2^2\right), \\
& \alpha_3=-2 \varphi_1 \varphi_2.
\end{aligned}
$$
Thus, each isotropic vector in $\mathbb{R}_1^3$ corresponds to a vector  in the twodimensional subspace of  $\mathbb{R}_1^3$ and vice versa. The vector $\varphi=\left(\varphi_1, \varphi_2\right) \cong\left[\begin{array}{l}\varphi_1 \\ \varphi_2\end{array}\right]$ obtained in this way is called a Hyperbolic spinor. The set of spinors is denoted by $\mathbb{S}$. For the spinor $\varphi$, by using the complex conjugate $\bar{\varphi}$ and the matrix $C=\left[\begin{array}{cc}0 & 1 \\ -1 & 0\end{array}\right]$, Cartan is defined spinor conjugate and Castillo and Barrales is defined spinor mate as follows, respectively.

Sir W.R. Hamilton discovered the quaternions in 1843 . Six years later, J. Cockle introduced split-quaternions which were called coquaternions by him.
Let $F$ be a arbitrary field (its characteristic is not two) and $F^{\prime}$ be multiplicative group of $F$. The quaternion algebra $\left(\frac{a, b}{F}\right)$ is defined to be the $F$ algebra on generators $i$ and $j$ with the relations
$$
i^2=a, j^2=b, i j=-j i .
$$
If we take $k:=i j$, then we have $k^2=-a b \in F^{\prime}$ and $i k=-k i=a j$, $k j=-j k=b i$.

In the case $a=b=-1$ and $F=\mathbb{R},\left(\frac{-1,-1}{\mathbb{R}}\right)$ is the ring of quaternions over the reals (for details, see [15]). Similarly the algebra $\left(\frac{-1,1}{R}\right)$ is ring of split quaternions over reals. Split quaternions are also called coquaternions, paraquaternions, anti-quaternions, pseudo-quaternions or hyperbolic quaternions. We show this ring as follows
$$
P_{\mathbb{R}}=\left\{\alpha=a+b i+c j+d k \mid a, b, c, d \in \mathbb{R}, i^2=-1, j^2=k^2=1, i j=k=-j i\right\}.
$$
The conjugate of split quaternion $\alpha=a+b i+c j+d k$ is defined by
$$
\bar{\alpha}=a+b i-c j-d k
$$
and the norm of $\alpha$ is given by
$$
\begin{aligned}
N(\alpha) & =\alpha \bar{\alpha} \\
& =a^2+b^2-c^2-d^2 .
\end{aligned}
$$

The split Jacobsthall quaternion ${SJQ_n} = {J_n} + i{J_{n + 1}} + j{J_{n + 2}} + k{J_{n + 3}}$ and the split Jacobsthal-Lucas ${SJLQ_n} = {JL_n} + i{JL_{n + 1}} + j{JL_{n + 2}} + k{JL_{n + 3}}$ are defined by Ya\u{g}mur, in \cite{yaugmur2019split}. These quaternions satisfy the following recurrence relations, respectively
\[{SJQ_{n + 2}} = 2{SPQ_{n }} + {SPQ_{n+1}}\]
and
\[{SJLQ_{n + 2}} = 2{SJLQ_{n }} + {SPJQ_{n+1}}\]
based on  the sequences of Jacobsthall and Jacobsthal-Lucas numbers.

\section{Hyperbolic Jacobsthal Spinors}
In this section, we will correspond to each split Jacobsthal quaternion, a hyperbolic spinor composed of two hyperbolic components, using the relationship between quaternions and spinors expressed by Vivarelli, \cite{vivarelli1984development}. We will call this spinor corresponding to the quaternion the Hyperbolic Jacobsthal spinor. In addition, we will introduce the Hyperbolic Jacobsthal-Lucas spinors with a similar approach. Moreover, we will examine the properties o Hyperbolicf Jacobsthal and Jacobsthal-Lucas spinors and obtain theorems and formulas that characterize these spinors.
Let the set of split Jacobsthal quaternions be $\JJ$. Let us define the transformation that gives the relation between spinors and quaternions as follows
$$
\begin{aligned}
& f:\JJ\rightarrow \SS \\
& f\left(J_{n}+i J_{n+1}+j J_{n+2}+k J_{n+3}\right)=\left[\begin{array}{c}
J_{n}+j J_{n+3} \\
-J_{n+1}+j J_{n+2}
\end{array}\right]=H S J_{n}
\end{aligned}
$$
where $\mathbf{i}, \mathbf{j}, \mathbf{k}$ coincide with basis vectors given for real quaternions. This transformation is linear and injective but it is not surjective. So, with the help of this transformation, we obtain a new sequence from Jacobsthal quaternions. We say that this new sequence is a Hyperbolic Jacobsthal spinor sequence. Note that for $n \geq 0$.
$$
H S J_{n+2}=H S J_{n+1}+2 H S J_{n}
$$
can be written. So, the Hyperbolic Jacobsthal spinor sequence is the linear recurrence sequence. Now, we give some algebraic definitions for Hyperbolic Jacobsthal spinors.
\begin{defn}
Let the conjugate of the split Jacobsthal quaternion $SJQ_{n}$ be $SJQ_{n}^{*}=J_{n}-i J_{n+1}-j J_{n+2}-k J_{n+3}$. So, from the correspondence, the Hyperbolic Jacobsthal spinor $HSJ_{n}^{*}$ corresponding to conjugate of split Jacobsthal quaternion is written by
$$
\begin{aligned}
f\left(SJQ_{n}^{*}\right) & =f\left(J_{n}-i J_{n+1}-j J_{n+2}-k J_{n+3}\right) \\
& =\left[\begin{array}{l}
J_{n}-j J_{n+3} \\
J_{n+1}-j J_{n+2}
\end{array}\right] \equiv H S J_{n}^{*}
\end{aligned}
$$
In addition, the following definitions can be written using the conjugate definitions of Cartan\cite{cartan1913groupes}, Castillo and Barrales, \cite{torres2004spinor}. The ordinary hyperbolic conjugate of Hyperbolic Jacobsthal spinor $HSJ_{n}$ is
$$
\overline{H S J_{n}}=\left[\begin{array}{c}
J_{n}-j J_{n+3} \\
-J_{n+1}-j J_{n+2}
\end{array}\right].
$$
Hyperbolic Jacobsthal spinor conjugate $\tilde{HSJ_{n}}=jC \overline{HSJ_{n}}$ of Hyperbolic Jacobsthal spinor $HSJ_{n}$ is
$$
H \tilde{S}_{n} J_{n}=\left[\begin{array}{c}
-J_{n+2}-j J_{n+1} \\
J_{n+3}-j J_{n}
\end{array}\right].
$$
The mate of Hyperbolic Jacobsthal spinor $\check{HSJ_{n}}=-C \overline{HSJ_{n}}$ is
$$
\check{H S J_{n}}=\left[\begin{array}{c}J_{n+1}+j J_{n+2} \\ J_{n}-j J_{n+3}\end{array}\right] 
$$
where $C=\left[\begin{array}{cc}0 & 1 \\ -1 & 0\end{array}\right].$
\end{defn}
\begin{thm}\label{hjsbinet}
The Binet formula for Hyperbolic Jacobsthal spinor sequence is as follows:
$$
H S J_{n}=\frac{1}{3}\left(2^{n}\left[\begin{array}{c}
1+8 j \\
4 j
\end{array}\right]-(-1)^{n}\left[\begin{array}{l}
1-j \\
1+j
\end{array}\right]\right).
$$
\end{thm}

\begin{proof}
The characteristic equation of the recurrence relation of Hyperbolic Jacobsthal spinor is $x^{2}-x-2=0$. The roots of this equation are $\alpha=2, \beta=-1$. The Binet's formula for Hyperbolic Jacobsthal spinor is as follows:
$$
H S J_{n}=H S_{A} \alpha^{n}+H S_{B} \beta^{n}
$$
If we write this equation for $n=0$ and $n=1$, we obtain the Binet formula as
$$
H S J_{n}=\frac{1}{3}\left(2^{n}\left[\begin{array}{c}
1+8 j \\
4 j
\end{array}\right]-(-1)^{n}\left[\begin{array}{c}
1-j \\
1+j
\end{array}\right]\right)
$$
\end{proof}
\begin{thm}\label{hjsgf}
 The generating function for Hyperbolic Jacobsthal spinor $HSJ_{n}$ is
$$
G_{HSJ_{n}}(x)=\frac{1}{1-x-2 x^{2}}\left(x\left[\begin{array}{c}
3 j \\
-1+2 j
\end{array}\right]-\left[\begin{array}{c}
1+8 j \\
-2+4 j
\end{array}\right]\right).
$$
\end{thm}
\begin{proof}
The generating function $G_{HSJ_{n}}(x)$ for Hyperbolic Jacobsthal spinor can be written as follows:
$$
G_{HSJ_{n}}(x)=\sum_{n=0}^{\infty} H S J_{n} x^{n}
$$
If the recurrence relation of Hyperbolic Jacobsthal spinor is considered, we have
$$
\sum_{n=0}^{\infty} H S J_{n+2} x^{n}=\sum_{n=0}^{\infty} H S J_{n+1} x^{n}+2 \sum_{n=0}^{\infty} H S J_{n} x^{n}
$$
Then, we can write
$$
\sum_{n=2}^{\infty} H S J_{n} x^{n-2}=\sum_{n=1}^{\infty} H S J_{n} x^{n-1}+2 \sum_{n=0}^{\infty} H S J_{n} x^{n}
$$
Since the first two terms of Hyperbolic Jacobsthal spinor sequence are
$$
H S J_{0}=\left[\begin{array}{c}
3 j \\
-1+j
\end{array}\right] \quad H S J_{1}=\left[\begin{array}{c}
1+5 j \\
-1+3 j
\end{array}\right]
$$
we get
$$
G_{HSJ_{n}}(x)=\frac{1}{1-x-2 x^{2}}\left(x\left[\begin{array}{c}
3 j \\
-1+2 j
\end{array}\right]-\left[\begin{array}{c}
1+8 j \\
-2+4 j
\end{array}\right]\right).
$$
with the help of necessary operations.
\end{proof}

\begin{thm}\label{hjssum1}
 For the Hyperbolic Jacobsthal spinor $HSJ_{n}$, we have
\begin{enumerate}
\item$\sum_{s=0}^t H S J_{n+s}=\frac{1}{2}\left[H S J_{n+t+2}-H S J_{n+1}\right]$,
\item$\sum_{s=1}^n H S J_S=\frac{1}{2}\left[H S J_{n+2}-H S J_2\right]$.
\end{enumerate}
\end{thm}

\begin{proof}
For the first identity, we can write
$$
\begin{aligned}
\sum_{s=0}^t H S J_{n+s}&=H S J_n+ H S J_{n+1}+\cdots+H S J_{n+t}\\
&=\left[\begin{array}{c}
J_n+j J_{n+3} \\
-J_{n+1}+j J_{n+2}
\end{array}\right]+\left[\begin{array}{c}
J_{n+1}+j J_{n+4} \\
-J_{n+2}+j J_{n+3}
\end{array}\right]+\cdots+\left[\begin{array}{l}
J_{n+t}+j J_{n+t+3} \\
-J_{n+t+1}+j J_{n+t+2}
\end{array}\right]\\
&=\frac{1}{2}\left[\begin{array}{c}
J_{n+t+2}-J_{n+1}+j\left(J_{n+2+5}-J_{n+4}\right) \\
-\left(J_{n+t+3}-J_{n+2}\right)+j\left(J_{n+t+4}-J_{n+3}\right)
\end{array}\right]\\
&=\frac{1}{2}\left(\left[\begin{array}{c}
J_{n+1+2}+j J_{n+t+5} \\
-J_{n+t+3}+j J_{n+t+4}
\end{array}\right]-\left[\begin{array}{l}
J_{n+1}+j J_{n+4} \\
-J_{n+k+4}+j J_{n+3}
\end{array}\right]\right)\\
&=\frac{1}{2}\left(H S J_{n+t+2}-H S J_{n+1}\right).
\end{aligned}
$$
The second identity is proved in a similar way as the second.
\end{proof}
\begin{thm}
 Let $n \geq 1, r \geq 1$ be any integers. Then,
\begin{enumerate}
\item$ H S J_{n+r}+H S J_{n-r}=\frac{1}{3}\left(\left(2^{n-r}+2^{n+r}\right)\left[\begin{array}{c}8+j \\ 2+4 j^j\end{array}\right]-2 \cdot(-1)^{n-1}\right) \left.\left[\begin{array}{l}1-j \\ 1+j\end{array}\right]\right)$,
\item$H S J_{n+r}-H S J_{n-r}=\frac{1}{3} 2^{n-r}\left(2^{2 r}-1\right)\left[\begin{array}{c}1+8 j \\ -2+4 j\end{array}\right]$.
\end{enumerate}
\end{thm}

\begin{proof}
 For the first identity, by using Binet formula for Hyperbolic Jacobsthal spinor, we can write
$$
\begin{aligned}
H S J_{n+r}+H S J_{n-r} & =\frac{1}{3} A\left(2^{n+r}+2^{n-r}\right)-\frac{1}{3} B\left((-1)^{n+r}+(-1)^{n-1}\right. \\
& =\frac{1}{3} A\left(2^{n+r}+2^{n-r}\right)-\frac{1}{3} B(-1)^{n-r}\left((-1)^{2 r}+1\right) \\
& =\frac{1}{3}\left(\left(2^{n-r}+2^{n+r}\right)\left[\begin{array}{l}
8+j \\
2+4 j
\end{array}\right]-2(-1)^{n-r}\left[\begin{array}{l}
1-j \\
1+j
\end{array}\right]\right)
\end{aligned}
$$
where $A=\left[\begin{array}{l}1+8j \\ 4 j\end{array}\right], B=\left[\begin{array}{c}1-j \\ 1+j\end{array}\right]$. 

The second identity is proved in a similar way as the first.

\end{proof}
\begin{thm}For $n \geq 1$, we have
$$
H S J_{n+1}+H S J_n=2^n\left[\begin{array}{c}
1+8 j \\
4 j
\end{array}\right].
$$
\end{thm}

\begin{proof}
We prove the theorem by induction on $n$. For $n=1, H S J_2+H S J_1=2\left[\begin{array}{c}1+8 j \\ 4 j\end{array}\right]$, then this equation is true. We suppose that equation is true for $n=k>1$. We will show that the equation is true for $k+1$.
$$
\begin{aligned}
H S J_{k+2}+H S J_{k+1}&=H S J_{k+1}+2 H S J_k+H S I_{k+1} \\
&=2\left(H S J_{k+1}+H S J_k\right)\\
&=2 \cdot 2^k\left[\begin{array}{c}
1+8 j \\
4 j
\end{array}\right]\\
&=2^{k+1}\left[\begin{array}{c}
1+8 j \\
4 j
\end{array}\right].
\end{aligned}
$$
\end{proof}
\begin{thm}
For the Hyperbolic Jacobsthal spinor $HSJ_n$, we have
\begin{enumerate}
\item$\sum_{i=1}^n H S J_{2 i}=\frac{2}{3} H S J_{2 n+1}+\frac{1}{3}\left[H S J_2-(2 n+1) H S J_3+n \cdot HS J_4\right]$,
\item$\sum_{i=1}^n H S J_{2 i-1}=\frac{2}{3} H S J_{2 n}-\frac{1}{3}\left[n H S J_4-2 n H S J_3+2 H S J_0\right]$.
\end{enumerate}
\end{thm}
\begin{proof}
This theorem is proved in a similar way as theorem \ref{hjssum1}.
\end{proof}

Now, we introduce the Hyperbolic Jacobsthal-Lucas spinors with a similar approach to Hyperbolic Jacobsthal spinors. Moreover, we examine the properties of Hyperbolic Jacobsthal-Lucas spinors and obtain theorems and formulas that characterize these spinors. Thus, Hyperbolic Jacobsthal-Lucas spinor $HSJL_{n}$ corresponding to split Jacobsthal-Lucas quaternion $SJLQ_{n}$ is

$$
\begin{aligned}
\chi\left(J L_n+i J L_{n+1}+j J L_{n+2}+k J L_{n+3}\right) & =\left[\begin{array}{l}
J L_n+j J L_{n+3} \\
-J L_{n+1}+j J L_{n+2}
\end{array}\right] \\
& =H S J L_n
\end{aligned}
$$
\begin{defn}
Let the conjugate of the split Jacobsthal-Lucas quaternion $SJLQ_{n}$ be $SJLQ_{n}^{*}=J L_n-i J L_{n+1}-j J L_{n+2}-k J L_{n+3}$. So, from the correspondence, the Hyperbolic Jacobsthal-Lucas spinor $HSJL_{n}^{*}$ corresponding to conjugate of split Jacobsthal-Lucas quaternion is written by
$$
\begin{aligned}
 X\left(SJLQ_n^*\right)&=X\left(J L_n-i J lL{n+1}-j JL_{n+2}-k J L_{n+3}\right) \\
& =\left[\begin{array}{l}
J L_n-j J L_{n+3} \\
J L_{n+1}-j J L_{n+2}
\end{array}\right]\\
&=H S J L_n^* 
\end{aligned}
$$
In addition, the following definitions can be written using the conjugate definitions of Cartan \cite{cartan1913groupes}, Castillo and Barrales, \cite{torres2004spinor}. The ordinary hyperbolic conjugate of Hyperbolic Jacobsthal-Lucas spinor $HSJL_{n}$ is
$$
\overline{H S J L_n}=\left[\begin{array}{c}
J L_n-j J L_{n+3} \\
-J L_{n+1}-j J l_{n+2}
\end{array}\right].
$$
Hyperbolic Jacobsthal-Lucas spinor conjugate $\tilde{HSJL_{n}}=jC \overline{HSJL_{n}}$ of Hyperbolic Jacobsthal-Lucas spinor $HSJL_{n}$ is
$$
\tilde{HS J L_n}=\left[\begin{array}{c}
-J L_{n+2}-j J L_{n+1} \\
J L_{n+3}-j JL_n
\end{array}\right].
$$
The mate of Hyperbolic Jacobsthal-Lucas spinor $\check{HSJL_{n}}=-C \overline{HSJL_{n}}$ is
$$
\check{H S JL_{n}}=\left[\begin{array}{c}JL_{n+1}+j JL_{n+2} \\ JL_{n}-j JL_{n+3}\end{array}\right] 
$$
where $C=\left[\begin{array}{cc}0 & 1 \\ -1 & 0\end{array}\right].$
\end{defn}
\begin{thm}
The Binet formula for the Hyperbolic Jacobsthal-Lucas spinor $HSJL_{n}$ is
$$
H S J L_n=2^n\left[\begin{array}{c}
1+8 j \\
-2+4 j
\end{array}\right]+(-1)^n\left[\begin{array}{l}
1-j \\
1+j
\end{array}\right].
$$
\end{thm}
\begin{proof}
This theorem is proved in a similar way as theorem \ref{hjsbinet}.
\end{proof}
\begin{thm}
 The generating function for Hyperbolic Jacobsthal-Lucas spinor $HSJL_{n}$ is
$$
G_{HSJ L_n}(x)=\frac{1}{1-x-2 x^2}\left(x\left[\begin{array}{c}
2+7 j \\
-1+5 j
\end{array}\right]-3\left[\begin{array}{c}
1+8 j \\
-2+4 j
\end{array}\right]\right).
$$
\end{thm}
\begin{proof}
This theorem is proved in a similar way as theorem \ref{hjsgf}.
\end{proof}

\begin{thm}
For the Hyperbolic Jacobsthal-Lucas spinor $HSJL_n$, we have
$$
\sum_{S=0}^t H S J L_{n+s}=\frac{1}{2}\left[H S J L_{n+t+2}-H S J L_{n+1}\right].
$$
\end{thm}
\begin{proof}
This theorem is proved in a similar way as theorem \ref{hjssum1}.
\end{proof}

\begin{thm}
Let $n \geq 1, r \geq 1$ be any integers. Then,
\begin{enumerate}
\item$H S J L_{n+1}+H S J L_n=3 \cdot 2^n\left[\begin{array}{c}1+8 j \\ -2+4 j\end{array}\right]$,
\item$H S J L_{n+r}-H S J L_{n-r}=3 \cdot 2^{n-1}\left[\begin{array}{c}1+8 j \\ -2+4 j\end{array}\right]$.
\end{enumerate}
\end{thm}
\begin{proof}
For the first identity, by using Binet formulas of Hyperbolic Jacobsthal-Lucas spinor, we can write
$$
\begin{aligned}
HSJL_{n+1}+ H S J L_n & =3.2^n A-(-1)^n B+(-1)^n B \\
& =3.2^n A \\
& =3.2^n\left[\begin{array}{c}
1+8 j \\
-2+4 j
\end{array}\right].
\end{aligned}
$$
where $A=\left[\begin{array}{l}1+8j \\ -2+4 j\end{array}\right], B=\left[\begin{array}{c}1-j \\ 1+j\end{array}\right]$. 

The second identity is proved in a similar way as the first.
\end{proof}
\begin{thm}
For $n \geq 1$, we have
\begin{enumerate}
\item$H S J L_n+H S J_n=2 H S J_{n+1}$,
\item$HSJL_{n}+3 H S J_n=2^{n+1}\left[\begin{array}{c}1+8 j \\ -2+4 j\end{array}\right]$.
\end{enumerate}
\end{thm}
\begin{proof}
For the second identity, by using Binet formulas both Jacobsthal and Jacobsthal-Lucas spinors, we have
$$
\begin{aligned}
H S J L_n+3 H S J_n & =2^n A+(-1)^n B+3 \cdot \frac{1}{3}\left(2^n A-(-1)^n B\right) \\
& =2^n A+(-1)^n B+2^n A-(-1)^n B \\
& =2^{n+1} A \\
& =2^{n+1}\left[\begin{array}{c}
1+8 j \\
-2+4 j
\end{array}\right]
\end{aligned}
$$
The first identity is proved in a similar way as the second.
\end{proof}

\begin{thm}
For $n \geq 1$, we have
$$
H S J_n J L_n+2 H S J_{n-1} J L_{n-1}=H J S L_{2 n-1}.
$$
\end{thm}

\begin{proof}
 By using Binet formulas for both Hyperbolic Jacobsthal spinor and Jacobsthal-Lucas numbers, we have
$$
\begin{aligned}
H S J_n J L_n+2 H S J_{n-1} JL_{n-1} & =\frac{1}{3}\left(2^n A-(-1)^n B\right)\left(2^n+(-1)^n\right) \\
& +\frac{2}{3}\left(2^{n-1} A-(-1)^{n-1} B\right)\left(2^{n-1}+(-1)^{n-1}\right) \\
& =\frac{1}{3}\left(2^{2 n}+2^{2 n-1}\right) A+\frac{1}{3}\left((-1)^{2 n}+2(-1)^{2 n-1}\right) B \\
= & \frac{1}{3} 3 \cdot 2^{2 n-1} A+\frac{1}{3} \cdot 3 \cdot(-1)^{2 n-1} B \\
= & 2^{2 n-1} A+(-1)^{2 n-1} B \\
= & H J S L_{2 n-1}.
\end{aligned}
$$
\end{proof}

\section{Hyperbolic Jacobsthal Polynomial Spinors}
The Jacobsthal polynomials, defined by $J_n(x)=J_{n-1}(x)+2xJ_{n-2}(x)$, where $J_0(x)=0, J_1(x)=1$, \cite{horadam1997jacobsthal}. In this section, we define the Hyperbolic Jacobsthal polynomial spinor as follows using the Jacobsthal polynomials sequence.
$$
H S J_{n}(x)=\left[\begin{array}{c}
J_{n}(x)+j J_{n+3}(x) \\
-J_{n+1}(x)+j J_{n+2}(x)
\end{array}\right]
$$
We say that this new sequence is a Hyperbolic Jacobsthal polynomial spinor sequence. Note that
$$
H S J_{n}(x)=H S J_{n-1}(x)+2x H S J_{n-2}(x)
$$
can be written. So, the Hyperbolic Jacobsthal polynomial spinor sequence is the linear recurrence sequence. 
\begin{thm}
The Binet formula for Hyperbolic Jacobsthal polynomial spinor sequence is as follows:
$$
H S J_{n}(x)=\frac{1}{2c}\left(A(x)\alpha(x)+B(x)\beta(x)\right),
$$
where $c=\sqrt{8 x+1}$ and \\ $A(x)=\left[\begin{array}{c}
-1+c+(4 x+1+c)j \\
-1-c+(4x+1+c)j
\end{array}\right]$,
$B(x)=\left[\begin{array}{c}
-2+(c(2x+1)-6 x-1)j \\
-1-c+(c-4x-1)j
\end{array}\right].$
\end{thm}

\begin{proof}
The characteristic equation of the recurrence relation of Hyperbolic Jacobsthal polynomial  spinor is $t^{2}-t-2x=0$. The roots of this equation are $\alpha(x)=\frac{1+\sqrt{8 x+1}}{2}, \beta(x)=\frac{1-\sqrt{8 x+1}}{2}$. The Binet's formula for Hyperbolic Jacobsthal polynomial spinor is as follows:
$$
H S J_{n}(x)=A(x)\alpha(x)+B(x)\beta(x).
$$
If we write this equation for $n=0$ and $n=1$, we obtain the Binet formula.
\end{proof}
\begin{thm}
 The generating function for Hyperbolic Jacobsthal  polynomial spinor $HSJ_{n}(x)$ is
$$
G_{HSJ_{n}}(t,x)=\frac{1}{1-t-2 x t^{2}}\left((1-t)\left[\begin{array}{c}
(2 x+1) j \\
-1+j
\end{array}\right]+\left[\begin{array}{c}
1+(4 x+1) j \\
-1+(2 x+1) j
\end{array}\right]\right).
$$
\end{thm}

\begin{proof}
It is proved in a similar way as the theorem 3.3.
\end{proof}
Similarly, the Hyperbolic Jacobsthal-Lucas polynomials spinor can be defined using the Jacobsthal-Lucas polynomials sequence. Additionally, Binet formula and generating function can be studied.

\section{Conclusion}
Spinors are mathematical objects that have important applications in physics and mathematics, particularly in the areas of quantum mechanics and relativity theory. A spinor can be thought of as an element of a spin representation, which is a mathematical structure that describes the behavior of quantum mechanical particles with intrinsic angular momentum, or spin. The study you mentioned explores the relationship between split Jacobsthal and split Jacobsthal-Lucas quaternions and spinors. Split Jacobsthal and split Jacobsthal-Lucas quaternions are a generalization of the well-known Jacobsthal and Jacobsthal-Lucas numbers to the quaternion algebra, which is a four-dimensional extension of the complex numbers. The study defines Hyperbolic Jacobsthal, Jacobsthal-Lucas and Jacobsthal polynomial spinor sequences using these quaternions. The study describes the recurrence relations of these sequences and provides various properties of these spinors, including their generator function and Binet formula. The generator function is a mathematical expression that generates the members of the sequence, while the Binet formula provides a closed-form expression for the$n$th member of the sequence. The study also derives some identities provided by these spinors. Overall, this study provides a new perspective on the relationship between quaternions and spinors, which may have implications for further research in quantum mechanics and relativity theory. Additionally, the study suggests that the use of split Jacobsthal and split Jacobsthal-Lucas quaternions could be beneficial in the study of other mathematical structures, such as Lie algebras and Lie groups. This demonstrates the potential for interdisciplinary research between different areas of mathematics and physics, and highlights the importance of understanding the connections between seemingly unrelated mathematical structures.
\\
{\bf{Data Availibility Statement}} No datasets were generated or analysed during the current study.\\ \\
{\bf{Funding}} This research received no external funding.\\ \\
{\bf{Declarations}}\\ \\
{\bf{Conflict of interest}} We declare that we have no conflict of interest.


\end{document}